\documentclass[11pt]{article}
\usepackage[english]{babel}
\usepackage{times}
\usepackage{amsmath,amstext,amssymb,amsthm, amsfonts}
\usepackage[dvips]{graphics,color}
\newcommand{\bi}{\begin{itemize}}  %begin itemize
\newcommand{\ei}{\end{itemize}}     %end itemize
\newcommand{\bc}{\begin{center}}  %begin center
\newcommand{\ec}{\end{center}}     %end center

\newcommand{\ls}[1]
   {\dimen0=\fontdimen6\the\font \lineskip=#1\dimen0
   \advance\lineskip.5\fontdimen5\the\font \advance\lineskip-\dimen0
   \lineskiplimit=.9\lineskip \baselineskip=\lineskip
   \advance\baselineskip\dimen0 \normallineskip\lineskip
   \normallineskiplimit\lineskiplimit \normalbaselineskip\baselineskip
   \ignorespaces }

\numberwithin{equation}{section}

\newtheorem{lemma}{Lemma}[section]
\newtheorem{theorem}[lemma]{Theorem}
\newtheorem{corollary}[lemma]{Corollary}

\newtheorem{example}[lemma]{Example}

\newtheorem{remark}[lemma]{Remark}

\newcommand{\ilim} {\mathop{\rm lim\,inf}}

\def\S{\mathbb{S}}

\def\R{\overline{\mathbb{R}}}
\def\M{\mathcal{M}}

\def\B{\mathcal{B}}

\title{Uniform Fatou's Lemma}

\begin{document}

%\begin{frontmatter}
\maketitle

\begin{center}
 Eugene~A.~Feinberg \footnote{Department of Applied Mathematics and
Statistics,
 State University of New York at Stony Brook,
Stony Brook, NY 11794-3600, USA, eugene.feinberg@sunysb.edu}, \
Pavlo~O.~Kasyanov \footnote{Institute for Applied System
Analysis, National Technical University of Ukraine ``Kyiv
Polytechnic Institute'', Peremogy ave., 37, build, 35, 03056, Kyiv,
Ukraine,\ kasyanov@i.ua}, and {Michael~Z.~Zgurovsky} \footnote{National Technical University of Ukraine ``Kyiv
Polytechnic Institute'', Peremogy ave., 37, build, 1, 03056, Kyiv,
Ukraine,\
 zgurovsm@hotmail.com}
 \end{center}


%\title{Converse Statements to Fatou's Lemma}
%\title{Uniform Fatou's Lemma for Converging Finite Measures}
%
%
%%% use optional labels to link authors explicitly to addresses:
%%% \author[label1,label2]{}
%%% \address[label1]{}
%%% \address[label2]{}
%
%\author[1]{Eugene~A.~Feinberg}
%\address[1]{Department of Applied Mathematics and Statistics,
% Stony Brook University,
%Stony Brook, NY 11794-3600, USA, eugene.feinberg@sunysb.edu}
%\author[2]{Pavlo~O.~Kasyanov} \address[2]{Institute for Applied System
%Analysis, National Technical University of Ukraine ``Kyiv
%Polytechnic Institute'', Peremogy ave., 37, build, 35, 03056, Kyiv,
%Ukraine,\ kasyanov@i.ua}
%\author[3]{Michael~Z.~Zgurovsky}
%\address[3]{National Technical University of Ukraine ``Kyiv
%Polytechnic Institute'', Peremogy ave., 37, build, 1, 03056, Kyiv,
%Ukraine,\
% zgurovsm@hotmail.com}
%%
%%\author{}
%
%\address{}

\begin{abstract}
%This note provides sufficient conditions for converging properties
%of functions $\{f^{(n)}\}_{n=1,2,\ldots}$ under assumptions that
%given measures $\{P^{(n)}\}_{n=1,2,\ldots}$ and integrals $\{\int
%f^{(n)}(s)P^{(n)}(ds)\}_{n=1,2,\ldots}$ converge.
%This note provides uniform and exact Fatou's lemmas for finite measures
%that converge in total variation. As a corollary the classic
%Fatou's lemma is improved up to a criterion. Results applied to ...

Fatou's lemma is a classic fact in real analysis that states that  the limit inferior of integrals of functions is greater than or equal to the integral of the inferior limit.
This paper introduces a stronger inequality that holds uniformly for integrals on measurable subsets of a measurable space. The necessary and sufficient condition, under which this inequality holds for a sequence of finite measures converging in total variation, is provided.{ This statement  is called the uniform Fatou's lemma, and it holds
under the minor assumption that all the integrals are well-defined. The uniform Fatou's lemma improves the classic Fatou's lemma in the following directions: the uniform Fatou's lemma states a more precise inequality, it provides the necessary and sufficient condition, and it deals with variable measures.  Various corollaries of the uniform Fatou's lemma are formulated.
The examples in this paper demonstrate that: (a) the uniform Fatou's lemma may indeed provide a more accurate inequality than the classic Fatou's lemma; (b) the uniform Fatou's lemma does not hold if convergence of measures in total variation is relaxed to setwise convergence.}
 %  a uniform version of Fatou's lemma for a
%sequence of finite measures  converging in total variation and establishes the necessary and sufficient condition for its validity. % of the uniform %Fatou's %lemma and
%% and some of %them also follow from classic results.
%A statement, that is more particular than one of the formulated corollaries, was originally introduced by the authors for the analysis of control problems with %incomplete observation.  That statement does not follow from classic results, and this observation motivated the current study.
\end{abstract}

%\begin{keyword}
%%% keywords here, in the form: keyword \sep keyword
%finite measure \sep convergence \sep Fatou's lemma \sep Radon-Nikodym derivative
%
%%% PACS codes here, in the form: \PACS code \sep code
%
%%% MSC codes here, in the form: \MSC code \sep code
%%% or \MSC[2008] code \sep code (2000 is the default)
%
%\end{keyword}

%\end{frontmatter}


%\section{}
%\label{}





\section{Introduction and Main Results}

%For a metric space $\S$, let ${\mathcal B}(\S)$ be its Borel
%$\sigma$-field, %EF10.5 on $\S$,
%that is, the $\sigma$-field generated by all open subsets of the
%metric space $\S$.  For a Borel set $E\in\B( \S),$ we denote by
%${\mathcal B}(E)$ the $\sigma$-field whose elements are
%intersections of $E$ with elements of ${\mathcal B}(\S)$.  Observe
%that $E$ is a metric space with the same metric as on $\S$, and
%${\mathcal B}(E)$ is its Borel $\sigma$-field. For a metric space
%For a measurable space $(\S,\Sigma)$, let  $\M(\S)$ denote the
%\textit{family of finite measures} on $(\S,\Sigma).$

%For a Borel subset $X$ of $\R,$ its Borel $\sigma$-field is denoted by ${\cal B}(X).$
%In this paper,

 Fatou's lemma is an important fact in real analysis that has significant applications in various fields.  It provides the inequality that relates the limit inferior of integrals of functions and the integral of the inferior limit. This paper introduces the uniform Fatou's lemma for a sequence of finite measures converging in total variation, describes the necessary and sufficient condition for the validity of this statement, and provides corollaries and counter-examples.

For a measurable space $(\S,\Sigma)$, let  $\M(\S)$ denote the
\textit{family of finite measures} on $(\S,\Sigma).$ Let $\mathbb{R}$ be a real line and
$\R:=\mathbb{R}\cup\{\pm\infty\}.$ A function $f:\S\to \R$ is called  measurable if $\{s\in\S:\, f(s)<\alpha\}\in\Sigma$ for
each $\alpha\in\mathbb{R}.$
%
%This convergence can be defined via a metric%
%$\rm dist(\cdot,\cdot)$  on $\M(\S)$ called the distance in total variation.
%
 For $\mu\in\M(\S)$ and a measure $\nu$ on $\S$, consider the \textit{distance in total variation}
\[%begin{equation}\label{eqdefdist}
%\begin{aligned}
{\rm dist}(\mu,\nu):=\sup\left\{|\int_\S f(s)\mu(ds) -\int_\S f(s)\nu(ds)|  :\
f:\S\to [-1,1]\mbox{ is measurable} \right.\bigg\}.%\end{aligned}
\]%end{equation}

We recall that a sequence of finite measures $\{\mu^{(n)}\}_{n=1,2,\ldots}$ on
$\S$ \textit{converges in total variation} to a measure $\mu$ on $\S$ if
$\lim_{n\to\infty}{\rm dist}(\mu^{(n)},\mu)=0.$ Of course, if a sequence of finite measures $\{\mu^{(n)}\}_{n=1,2,\ldots}$ on
$\S$ converges in total variation to a measure $\mu$ on $\S$, then $\mu\in \M(\S)$ and $\mu^{(n)}(\S)\to \mu(\S)$ as $n\to\infty$.
%We are interested in studying convergence in total variation of finite measures.

%A sequence of finite measures $\{ \mu^{(n)}\}_{n=1,2,\ldots}$
%on $\S$ \textit{converges in    variation} to measure $
%\mu$ on $ \S$ if
%\[
%\sup\left\{\int_\S f(s) \mu^{(n)}(ds)-\int_\S f(s) \mu(ds)\, : \, f:\S\to
%[-1,1]\mbox{ is measurable} \right\}\to 0, \ n\to\infty.
%\]
%Let us set $\R:=\mathbb{R}\cup\{\pm\infty\}$.
For %$(\S,\Sigma)$ be a measurable space and
$\mu\in \M(\S)$ consider
the vector space $L^1(\S;\mu)$ of all  measurable functions
$f:\S\to\R$, whose absolute values have finite integrals, that is,
$\int_\S |f(s)| \mu(ds) < +\infty$.
%\begin{equation}\label{e:condint} \min\{ \int_\S f^+(s) \mu(ds),
%
The following theorem is the main result of this paper.

\begin{theorem}{\rm(Necessary and Sufficient Conditions for Uniform Fatou's Lemma for Variable Measures and Unbounded Below Functions)}
%{\rm (Exact Fatou's lemma for probability measures that converges in total variation)}
\label{cfl}
Let $(\S,\Sigma)$ be a measurable space,  $\{
\mu^{(n)}\}_{n=1,2,\ldots}\subset \M(\S)$ {converge} in  total
variation to a measure $ \mu$ on $\S$,  $f\in L^1(\S;\mu)$, and   $f^{(n)}
\in L^1(\S;\mu^{(n)})$ for each $n=1,2,\ldots\,.$
%If %all the
%integrals $\int_{S}f(s) P^{(n)}(ds)$, $\int_{S}f^{(n)}(s) P^{(n)}(ds)$, $n=1,2,\ldots$, and $\int_{S}f(s) P(ds)$ are defined, and
%\begin{equation}\label{eq:cfl0}
%\slim\limits_{n\to\infty}\sup\limits_{S\in\Sigma}\left(\int_{S}f(s) P^{(n)}(ds)- \int_{S}f(s) P(ds)\right)\le 0,
%\end{equation}
Then the  inequality
\begin{equation}\label{eq:cfl1}
\ilim\limits_{n\to\infty}\inf\limits_{S\in\Sigma}\left(\int_{S}f^{(n)}(s)
\mu^{(n)}(ds)- \int_{S}f(s) \mu(ds)\right)\ge 0
\end{equation}
holds if and only if the following two statements hold:
\begin{enumerate}
\item[(i)] for each $\varepsilon>0$
\begin{equation}\label{eq:cfl2}
 \mu(\{s\in\S\,:\, f^{(n)}(s)\le f(s)-\varepsilon\})\to 0\mbox{ as }n\to\infty,
\end{equation}
and, therefore, there exists a subsequence $\{f^{(n_k)}\}_{k=1,2,\ldots}\subseteq\{f^{(n)}\}_{n=1,2,\ldots}$
%$\{n_k\}_{k=1,2,\ldots}$
such that
\begin{equation}\label{eq:cfl3}
\ilim\limits_{k\to\infty}f^{(n_k)}(s)\ge f(s)\quad  \mbox{ for }\mu\mbox{-a.e.
}s\in\S;
\end{equation}
\item[(ii)] the inequality
\begin{equation}\label{eq:ui}
\ilim\limits_{K\to+\infty}\inf\limits_{n=1,2,\ldots
}\int_{\S}f^{(n)}(s){\bf I}\{s\in\S\,:\, f^{(n)}(s)\le -K\}
\mu^{(n)}(ds)\ge 0
\end{equation}
holds.
\end{enumerate}
\end{theorem}

%\begin{remark}\label{re:cfl}
%{\rm If statement (\ref{eq:cfl2}) holds for each $\varepsilon>0$,
%then inequality (\ref{eq:cfl3}) does not hold for the whole sequence
%$\{f^{(n)}\}_{n=1,2,\ldots}$, in general, even for uniformly bounded
%functions $\{f^{(n)}\}_{n=1,2,\ldots}$ that satisfy inequalities
%(\ref{eq:cfl1}) and (\ref{eq:ui}). Indeed, if $\S=[0,1]$, $\Sigma$
%is the Borel $\sigma$-field on $\S$, $P^{(n)}=P$ is the Lebesgue
%measure on $\S$, $f\equiv0$, and $f^{(n)}(s)=-{\bf I}\{s\in
%[\frac{j}{2^k}, \frac{j+1}{2^k}]\}$, where $k=[\log_2n]$, $j=n-2^k$,
%$s\in\S$, and $n=1,2,\ldots,$ then statement (\ref{eq:cfl2}) holds
%for each $\varepsilon>0$. Moreover,
%%\[
%%\inf\limits_{S\in\Sigma}\left(\int_{S}f^{(n)}(s) P^{(n)}(ds)-
%%\int_{S}f(s) P(ds)\right)\ge-\frac1n \to 0, \ n\to\infty,
%%\]
%inequalities (\ref{eq:cfl1}) and (\ref{eq:ui}) hold. But,
%$\ilim\limits_{k\to\infty}f^{(n)}(s)=-1< 0=f(s)$ $P\mbox{-a.s. in
%}s\in\S.$ Therefore, (\ref{eq:cfl3}) does not hold for the whole
%sequence $\{f^{(n)}\}_{n=1,2,\ldots}$.
%}%If the absolute value $|f|$ of a function $f$ from
%%Theorem~\ref{cfl} is bounded above by a constant $C>0$  on $\S$,
%%then assumption (\ref{eq:cfl0}) holds, because the sequence $\{
%%P^{(n)}\}_{n=1,2,\ldots}\subset \M(\S)$ converges in  total
%%variation to
%%$ P\in\M(\S)$. Moreover, $f\in L^1(\S;P)\cap L^1(\S;P^{(n)})$ for each $n=1,2,\ldots$.}
%\end{remark}

\begin{remark}\label{23}{\rm Let $(\S,\Sigma)$ be a measurable space, $\{f^{(n)},f\}_{n=1,2,\ldots}
$ be a sequence of measurable functions, $ \mu  $ be a measure on $\S$. We note that if (\ref{eq:cfl2}) holds for
each $\varepsilon>0$, then (\ref{eq:cfl3}) holds; see Lemma~\ref{lem:1}.}
\end{remark}

We recall that the classic  Fatou's lemma can be formulated in the following form. \medskip\\
{\bf Fatou's lemma.} {\it  Let $(\S,\Sigma)$ be a measurable space, $\mu$ be a measure on $(\S,\Sigma)$ and $\{f,f^{(n)}\}_{n=1,2,\ldots}$ be a sequence of measurable nonnegative functions. Then the {inequality}
\begin{equation}\label{eq:FLcl1}
\ilim\limits_{n\to\infty} f^{(n)}(s)\ge f(s)\mbox{ for $\mu$-a.e. } s\in \S
\end{equation}
{implies}
\begin{equation}\label{eq:FLcl2}
\ilim\limits_{n\to\infty} \int_\S f^{(n)}(s)\mu(ds)\ge \int_\S f(s)\mu(ds).
\end{equation}}

{Note that there are generalizations of Fatou's lemma to functions that can take negative values.  For example, the conclusions of Fatou's lemma hold if all the functions have a common integrable minorant.}

We recall that a sequence of
measures $\{ \mu^{(n)}\}_{n=1,2,\ldots}$ from $\M(\S)$
\textit{converges setwise (weakly)} to $ \mu\in\M(\S)$ if for each bounded
 measurable (bounded continuous) function $f$ on $\S$
\[\int_\S f(s) \mu^{(n)}(ds)\to \int_\S f(s) \mu(ds) \qquad {\rm as \quad
}n\to\infty.
\]
%Under the assumptions of Fatou's lemma,
{If} $\{f,f^{(n)}\}_{n=1,2,\ldots}\subset L^1(\S;\mu)$, then for $\mu^{(n)}=\mu$, $n=1,2,\ldots,$ inequality (\ref{eq:cfl1}) is the uniform version of  inequality (\ref{eq:FLcl2}) of the Fatou's lemma.
There are generalized versions of Fatou's lemmas for weakly and setwise converging  sequences of measures;
see Royden~\cite[p. 231]{Ro}, Serfoso~\cite{Se}, Feinberg et al. \cite{TVP}, and references in \cite{TVP}. %, but for uniform
 %statement of Fatou's lemma it is necessary to consider the sequence of finite measures that converge in  total
%variation (see Example~\ref{exa2}).
 Theorem~\ref{cfl} provides necessary and sufficient conditions
for inequality (\ref{eq:cfl1}), when variable finite measures $\{\mu^{(n)}\}_{n=1,2,\ldots}$ converge
 in total variation to $\mu$. % see also Examples~\ref{exa0} and \ref{exa2}.
We note that: (a) inequality (\ref{eq:FLcl1}) implies statement (i) from Theorem~\ref{cfl}, but not vice versa (see Example~\ref{exa0});  (b) inequality  (\ref{eq:ui}) always holds for nonnegative functions $\{f^{(n)}\}_{n=1,2,\ldots}$; and (c) the assumption that  the convergence of $\{
\mu^{(n)}\}_{n=1,2,\ldots}\subset \M(\S)$ to a measure $ \mu$ on $\S$ takes place in  total
variation  is essential and cannot be relaxed to setwise convergence; see Examples~{\ref{exa2} -- \ref{exa3}}.

Theorem~\ref{cfl} implies the following two corollaries.

  %
%  (a) statement (\ref{eq:FLcl1}) imply convergence (\ref{eq:cfl2}) for each $\varepsilon>0$, but not vice versa; (b)
%inequality (\ref{eq:cfl1}) is the uniform version of inequality (\ref{eq:FLcl2}); (c) inequality (\ref{eq:ui}) holds.
%
%
%(\ref{eq:cfl1})

%Let $(\S,\Sigma)$ be a measurable space. If all measures $\mu^{(n)}\in\M(\S)$, $n=1,2,\ldots,$ are equal to $\mu\in \M(\S)$
%and $\{f,f^{(n)}\}_{n=1,2,\ldots}\subset L^1(\S;\mu)$ is the sequence of nonnegative functions, then (a)
%
% $\{f,f^{(n)}\}_{n=1,2,\ldots}\subset L^1(\S;\mu)$, then inequality (\ref{eq:FLcl1})

\begin{corollary}{\rm(Necessary and Sufficient Conditions for
Uniform Fatou's Lemma for Variable Measures and Nonnegative Functions)}%{\rm (Exact Fatou's lemma for probability measure)}
\label{cor:cflnongen}
Let $(\S,\Sigma)$ be a measurable space, the sequence $\{
\mu^{(n)}\}_{n=1,2,\ldots}\subset \M(\S)$ {converge} in  total
variation to a measure $ \mu$ on $\S$, $f^{(n)}
\in L^1(\S;\mu^{(n)})$, $n=1,2,\ldots,$ be  nonnegative functions and $f\in L^1(\S;\mu)$.     Then inequality (\ref{eq:cfl1}) holds
%\begin{equation}\label{eq:cor:cfl1}
%\ilim\limits_{n\to\infty}\inf\limits_{S\in\Sigma}\left(\int_{S}f^{(n)}(s)
%P(ds)- \int_{S}f(s) P(ds)\right)\ge 0,
%\end{equation}
if and only if statement (i) from Theorem~\ref{cfl} takes place. %; and b) the inequality holds:
%\begin{equation}\label{eq:uic1}
%\ilim\limits_{K\to+\infty}\inf\limits_{n=1,2,\ldots
%}\int_{\S}f^{(n)}(s){\bf I}\{s\in\S\,:\, f^{(n)}(s)\le -K\}
%P(ds)\ge 0.
%\end{equation}%, and, therefore,
%there is a subsequence $\{n_k\}_{k=1,2,\ldots}$ such that
%(\ref{eq:cfl3})
%hold.
\end{corollary}

\begin{corollary}{\rm(Necessary and Sufficient Conditions for Uniform Fatou's Lemma for Unbounded Below Functions)}%{\rm (Exact Fatou's lemma for probability measure)}
\label{cor:cfl}
Let $(\S,\Sigma)$ be a measurable space,  $\mu\in \M(\S)$, and  $\{f,\,
f^{(n)}\}_{n=1,2,\ldots}\subset L^1(\S;\mu)$. Then the inequality
\begin{equation}\label{eq:cor:cfl1}
\ilim\limits_{n\to\infty}\inf\limits_{S\in\Sigma}\left(\int_{S}f^{(n)}(s)
\mu(ds)- \int_{S}f(s) \mu(ds)\right)\ge 0\end{equation}
holds if and only if  statement (i) from Theorem~\ref{cfl} takes place and
\begin{equation}\label{eq:uic1}
\ilim\limits_{K\to+\infty}\inf\limits_{n=1,2,\ldots
}\int_{\S}f^{(n)}(s){\bf I}\{s\in\S\,:\, f^{(n)}(s)\le -K\}
\mu(ds)\ge 0.
\end{equation}%, and, therefore,
%there is a subsequence $\{n_k\}_{k=1,2,\ldots}$ such that
%(\ref{eq:cfl3})
%hold.
\end{corollary}
\begin{remark}\label{rem -}
{\rm

For each $a\in\mathbb{R}$ we denote $a^+:=\max\{a,0\}$ and $a^-:=a^+-a$. Note that $a=a^+-a^-$ and
$|a|=a^++a^-$.
%Under {the} assumptions of Corollary~\ref{cor:cfl}
For a measure $\mu$ on $\S$ and functions $f,g\in  L^1(\S;\mu),$
\[
\inf\limits_{S\in\Sigma}\left(\int_{S}g(s)
\mu(ds)- \int_{S}f(s) \mu(ds)\right)=-\int_{\S}(g(s)- f(s) )^-\mu(ds)\le 0.
\]
%$n=1,2,\ldots$.
Therefore, inequality (\ref{eq:cor:cfl1}) is equivalent to
\[
\lim\limits_{n\to\infty}\int_{\S}(f^{(n)}(s)- f(s) )^-\mu(ds)= 0.
\]
}\end{remark}

Each of the Corollaries~\ref{cor:cflnongen} and \ref{cor:cfl} implies the following statement.

\begin{corollary}{\rm(Necessary and Sufficient Conditions for Uniform Fatou's Lemma for Nonnegative Functions)}%{\rm (Exact Fatou's lemma for probability measure)}
\label{cor:cflnon}
Let  $(\S,\Sigma)$ be a measurable space,  $\mu\in \M(\S)$,  $f\in L^1(\S;\mu)$, and  $\{f^{(n)}\}_{n=1,2,\ldots}\subset L^1(\S;\mu)$ be a sequence of nonnegative functions. Then
inequality (\ref{eq:cor:cfl1}) holds
%\begin{equation}\label{}
%\ilim\limits_{n\to\infty}\inf\limits_{S\in\Sigma}\left(\int_{S}f^{(n)}(s)
%P(ds)- \int_{S}f(s) P(ds)\right)\ge 0,
%\end{equation}
if and only if statement (i) from Theorem~\ref{cfl} takes place.
%
%inequality (\ref{eq:cfl1}) holds
%%\begin{equation}\label{eq:cor:cfl1}
%%\ilim\limits_{n\to\infty}\inf\limits_{S\in\Sigma}\left(\int_{S}f^{(n)}(s)
%%P(ds)- \int_{S}f(s) P(ds)\right)\ge 0,
%%\end{equation}
%if and only if for each $\varepsilon>0$ formula (\ref{eq:cfl2}) holds. %; and b) the inequality holds:
%\begin{equation}\label{eq:uic1}
%\ilim\limits_{K\to+\infty}\inf\limits_{n=1,2,\ldots
%}\int_{\S}f^{(n)}(s){\bf I}\{s\in\S\,:\, f^{(n)}(s)\le -K\}
%P(ds)\ge 0.
%\end{equation}%, and, therefore,
%there is a subsequence $\{n_k\}_{k=1,2,\ldots}$ such that
%(\ref{eq:cfl3})
%hold.
\end{corollary}
\begin{remark}\label{re:cfla1111}
{\rm
Under the assumptions of Corollary~\ref{cor:cflnon},  inequality (\ref{eq:cor:cfl1}) is equivelent to $(f^{(n)}- f)^- \mathop{\to}\limits^\mu 0$ as $n\to\infty.$ This follows from   Remark~\ref{rem -}, the dominated convergence theorem, Chebyshev's inequality,  and because each function $(f^{(n)}- f)^-$, $n=1,2,\ldots,$ is majorated above by $f^+\in L^1(\S,\mu)$. Statement (i) from Theorem~\ref{cfl} holds if and only if $(f^{(n)}- f)^- \mathop{\to}\limits^\mu 0$, $n\to\infty.$ Therefore,  Corollary~\ref{cor:cflnon} also follows from classic results. Furthermore, the assumption, that the measure $\mu$ is finite, can be
omitted from Corollary~\ref{cor:cflnon}.}

\end{remark}

\begin{corollary}\label{cdct}{\rm(Necessary and Sufficient Conditions for Uniform Dominated Convergence Theorem for Variable Measures)}
%Let $(\S,\Sigma)$ be an arbitrary measurable space, $\{f,\, f^{(n)}\}_{n=1,2,\ldots}$ be %a sequence of
% measurable uniformly bounded real-valued functions on $\S$, $\{ P^{(n)}\}_{n=1,2,\ldots}\subset \M(\S)$ converges in  total  variation to
%$ P\in\M(\S)$, and
Let $(\S,\Sigma)$ be a measurable space,  the sequence $\{
\mu^{(n)}\}_{n=1,2,\ldots}\subset \M(\S)$ {converge} in  total
variation to a measure $ \mu$ on $\S$,  $f\in L^1(\S;\mu)$, and   $f^{(n)}
\in L^1(\S;\mu^{(n)})$ for each $n=1,2,\ldots\,.$
Then the equality
\begin{equation}\label{eq:cdct1}
\lim\limits_{n\to\infty}\sup\limits_{S\in\Sigma}\left|\int_{S}f^{(n)}(s) \mu^{(n)}(ds)-
\int_{S}f(s) \mu(ds)\right|= 0
\end{equation}
%Then the following inequality holds
%\begin{equation}\label{eq:cfl1asasas}
%\ilim\limits_{n\to\infty}\inf\limits_{S\in\Sigma}\left(\int_{S}f^{(n)}(s)
%\mu^{(n)}(ds)- \int_{S}f(s) \mu(ds)\right)\ge 0,
%\end{equation}
holds if and only if the following two statements hold:
\begin{enumerate}
\item[(i)] the sequence $\{f^{(n)}\}_{n=1,2,\ldots}$ converges in measure $ \mu $ to
$f$, and, therefore, there is a subsequence $\{f^{(n_k)}\}_{k=1,2,\ldots}\subseteq\{f^{(n)}\}_{n=1,2,\ldots}$
that  converges $ \mu$-a.e. to $f;$
\item[(ii)] the following equality holds:
\begin{equation}\label{eq:uicorleb}
\lim\limits_{K\to+\infty}\sup\limits_{n=1,2,\ldots
}\int_{\S} |f^{(n)}(s)|{\bf I}\{s\in\S\,:\, |f^{(n)}(s)|\ge K\}
\mu^{(n)}(ds)= 0.
\end{equation}
\end{enumerate}
\end{corollary}

{ We remark that,} %the necessary part of Corollary~\ref{cdct} is proved in Feinberg et al.
%\cite[Theorem~5.5]{POMDP} for probability measures $\{
%\mu^{(n)},\, \mu\}_{n=1,2,\ldots}$ and uniformly bounded %$\mathbb{R}$-valued
%measurable functions $\{f^{(n)}\}_{n=1,2,\ldots}$ defined on a standard Borel space $\S$.
{for uniformly} bounded functions $\{f^{(n)}\}_{n=1,2,\ldots},$ condition~(ii) {from} Corollary~\ref{cdct} always holds and therefore is not needed.
{The necessary part of Corollary~\ref{cdct} for probability measures $\{
\mu^{(n)},\, \mu\}_{n=1,2,\ldots}$ and uniformly bounded
measurable functions $\{f^{(n)},\, f\}_{n=1,2,\ldots},$ defined on a standard Borel space $\S,$   was introduced in { Feinberg et al.~\cite[Theorem~5.5]{POMDP}.  {This necessary condition}  was used {in Feinberg et al.~\cite{POMDP, POMDP2}} for  the analysis of control problems with incomplete observations, and it can be interpreted as a converse to a version of Lebesgue's dominated convergence theorem for a sequence of measures converging in total variation. The understanding of{ Feinberg et al.} \cite[Theorem~5.5]{POMDP} was the starting point for formulating and investigating the uniform Fatou's lemma.}

\begin{corollary}\label{cor:cdct}{\rm(Necessary and Sufficient Conditions for Uniform Dominated Convergence Theorem)}
Let $(\S,\Sigma)$ be a measurable space,  $\mu\in \M(\S)$, and  $\{f,\,
f^{(n)}\}_{n=1,2,\ldots}\subset L^1(\S;\mu)$. Then the equality
\begin{equation}\label{cor:eq:cdct1}
\lim\limits_{n\to\infty}\sup\limits_{S\in\Sigma}\left|\int_{S}f^{(n)}(s)\mu(ds)-\int_{S}f(s)
\mu(ds)\right|= 0
\end{equation}
holds if and only if
%\begin{enumerate}
%\item[(i)] %$\{f^{(n)}\}_{n=1,2,\ldots}$ converges in probability $ \mu $ to
%$f$, and, therefore, there is a subsequence $\{n_k\}_{k=1,2,\ldots}$
%such that $\{f^{(n_k)}\}_{k=1,2,\ldots}$ converges $ \mu$-a.s. to $f;$
statement (i) from Corollary~\ref{cdct} holds and
%\item[(ii)]
the sequence $\{f^{(n)}\}_{n=1,2,\ldots}$ is uniformly integrable, that is,
\begin{equation}\label{eq:uieqLebP}
\lim\limits_{K\to+\infty}\sup\limits_{n=1,2,\ldots
}\int_{\S} |f^{(n)}(s)|{\bf I}\{s\in\S\,:\, |f^{(n)}(s)|\ge K\}
\mu(ds)= 0.
\end{equation}
%\end{enumerate}
\end{corollary}
\begin{remark}\label{rem -Leb}
{\rm Under the assumptions of Corollary~\ref{cor:cdct}
\[
\begin{aligned}
\sup\limits_{S\in\Sigma}&\left|\int_{S}f^{(n)}(s)\mu(ds)-\int_{S}f(s)
\mu(ds)\right|\\
&=\max\left\{\int_{\S}(f^{(n)}(s)- f(s) )^-\mu(ds),\int_{\S}(f^{(n)}(s)- f(s) )^+\mu(ds) \right\}\ge 0.
\end{aligned}
\]
Therefore,  equality (\ref{cor:eq:cdct1}) is equivalent to
\[
\lim\limits_{n\to\infty}\int_{\S}|f^{(n)}(s)-f(s)|
\mu(ds)= 0,\]
and Corollary~\ref{cor:cdct} coincides with the classic  criterion of
strong convergence in  $L^1(\S;\mu)$.}
\end{remark}

%
%\begin{remark}
%{\rm In view of the Hahn's decomposition, assumption
%(\ref{cor:eq:cdct1}) is equivalent to the convergence of
%$\{f^{(n)}\}_{n=1,2,\ldots}$  to  $f$ in $L^1(\S;\mu)$; cf. Shiryaev \cite[p.~360]{Sh}. Therefore, Corollary~\ref{cor:cdct} coincides with the classic  criterion of
%strong convergence in  $L^1(\S;\mu)$.}
% \end{remark}

{The following two corollaries describe the relation between convergence properties of a sequence of finite signed measures $\{ { \Tilde \mu}^{(n)}\}_{n=1,2,\ldots}$ and the sequence of their Radon-Nikodym derivatives $\{ \frac{d{ \Tilde \mu}^{(n)}}{d\mu^{(n)}}\}_{n=1,2,\ldots}$ with respect to  finite measures $\{\mu^{(n)}\}_{n=1,2,\ldots}$  converging in total variation.} %The version of Corollary~\ref{cor:RNeq} %playing important role for studding control problems with incomplete information and filtration problems; Feinberg et al. \cite{POMDP, POMDP2}; Shiryaev %\cite{Sh} etc.

\begin{corollary}\label{cor:RNin}
Let $(\S,\Sigma)$ be a measurable space,  $\{\mu^{(n)}\}_{n=1,2,\ldots}\subset \M(\S)$, {$\mu$ be a measure on $\S$,}  and $\{{ \Tilde \mu}, { \Tilde \mu}^{(n)}\}_{n=1,2,\ldots}$ be a sequence of finite signed measures on $\S$. Assume that ${ \Tilde \mu}\ll\mu$ and ${ \Tilde \mu}^{(n)}\ll\mu^{(n)}$ for each $n=1,2,\ldots\,. $ If the sequence $\{
\mu^{(n)}\}_{n=1,2,\ldots}$ converges in  total
variation to $ \mu$, then the  inequality
\[
\ilim\limits_{n\to\infty}\inf\limits_{S\in\Sigma}\left({ \Tilde \mu}^{(n)}(S)- { \Tilde \mu}(S)\right)\ge 0
\]
holds if and only if the following two statements hold:
\begin{enumerate}
\item[(i)] for each $\varepsilon>0$
\[
 \mu(\{s\in\S\,:\, \frac{d{ \Tilde \mu}^{(n)}}{d\mu^{(n)}}(s)\le \frac{d{ \Tilde \mu}}{d\mu}(s)-\varepsilon\})\to 0\mbox{ as }n\to\infty,
\]
and, therefore, there exists a subsequence $\{\frac{d{ \Tilde \mu}^{(n_k)}}{d\mu^{(n_k)}}\}_{k=1,2,\ldots}\subseteq\{\frac{d{ \Tilde \mu}^{(n)}}{d\mu^{(n)}}\}_{n=1,2,\ldots}$
%sequence $\{n_k\}_{k=1,2,\ldots}$
such that
\[
\ilim\limits_{k\to\infty}\frac{d{ \Tilde \mu}^{(n_k)}}{d\mu^{(n_k)}}(s)\ge \frac{d{ \Tilde \mu}}{d\mu}(s)\quad  \mbox{ for }\mu\mbox{-a.e.
}s\in\S;
\]
\item[(ii)] {the  inequality}
\[
\ilim\limits_{K\to+\infty}\inf\limits_{n=1,2,\ldots
}{ \Tilde \mu}^{(n)}(\{s\in\S\,:\, \frac{d{ \Tilde \mu}^{(n)}}{d\mu^{(n)}}(s)\le -K\})
\ge 0.
\]
{holds.}
\end{enumerate}
\end{corollary}

%\begin{remark}\label{rem:sm}
We remark that, if $\{\Tilde  \mu^{(n)}\}_{n=1,2,\ldots}\subset \M(\S)$, then statement (ii) of Corollary~\ref{cor:RNin}  always holds because %a measure of any set is nonnegative.
$\Tilde  \mu^{(n)}(\cdot)\ge 0$ for all $n=1,2,\ldots\ .$
%Therefore, Corollary~\ref{cor:RNin} implies the following statement.
Corollary~\ref{cor:RNin} implies the following necessary and sufficient condition for the convergence in total variation of finite signed measures  $\{\Tilde  \mu^{(n)}\}_{n=1,2,\ldots}.$
%\end{remark}

\begin{corollary}\label{cor:RNeq}
Let $(\S,\Sigma)$ be a measurable space,  $\{\mu^{(n)}\}_{n=1,2,\ldots}\subset \M(\S)$, {$\mu$ be a measure on $\S$,}  and $\{{ \Tilde \mu}, { \Tilde \mu}^{(n)}\}_{n=1,2,\ldots}$ be a sequence of finite signed measures on $\S$. Assume that ${ \Tilde \mu}\ll\mu$ and ${ \Tilde \mu}^{(n)}\ll\mu^{(n)}$ for each $n=1,2,\ldots\,. $ If the sequence $\{
\mu^{(n)}\}_{n=1,2,\ldots}$ converges in  total
variation to $ \mu$, then the sequence $\{\Tilde
\mu^{(n)}\}_{n=1,2,\ldots}$ converges in  total
variation to $\Tilde \mu$, that is,
\[
\lim\limits_{n\to\infty}\sup\limits_{S\in\Sigma}\left|{ \Tilde \mu}^{(n)}(S)- { \Tilde \mu}(S)\right|= 0,
\]
%holds
if and only if the following  two statements hold:
\begin{enumerate}
\item[(i)] the sequence $\{\frac{d{ \Tilde \mu}^{(n)}}{d\mu^{(n)}}\}_{n=1,2,\ldots}$ converges in measure $ \mu $ to
$\frac{d{ \Tilde \mu}}{d\mu}$, and, therefore, there exists a subsequence $\{\frac{d{ \Tilde \mu}^{(n_k)}}{d\mu^{(n_k)}}\}_{k=1,2,\ldots}\subseteq\{\frac{d{ \Tilde \mu}^{(n)}}{d\mu^{(n)}}\}_{n=1,2,\ldots}$
that  converges $ \mu$-a.e. to $\frac{d{ \Tilde \mu}}{d\mu};$
\item[(ii)] the following inequality holds:
\[
\lim\limits_{K\to+\infty}\sup\limits_{n=1,2,\ldots
} { |\Tilde \mu^{(n)}|}(\{s\in\S\,:\, |\frac{d{ \Tilde \mu}^{(n)}}{d\mu^{(n)}}(s)|\ge K\})
= 0,
\]
{ where $|\Tilde \mu^{(n)}|(S)=\int_S |\frac{d{ \Tilde \mu}^{(n)}}{d\mu^{(n)}}(s)|\mu^{(n)}(ds)$, $S\in \Sigma.$}
\end{enumerate}
\end{corollary}

\section{Proofs}\label{sec2}

For a measurable function $g:\S\to\R$, real
number $K$, and set $S\in \Sigma$, we
denote:
\[
S_{g\ge K}:=\{s\in S\,:\, g(s)\ge K\}, \quad S_{g>K}:=\{s\in S\,:\,
g(s)> K\},
\]
\[
S_{g\le K}:=\{s\in S\,:\, g(s)\le K\}, \quad S_{g<K}:=\{s\in S\,:\,
g(s)< K\}.
\]

The proof of Theorem~\ref{cfl} consists of four auxiliary lemmas.

\begin{lemma}\label{lem:1} Let $(\S,\Sigma)$ be a measurable space, $\{f^{(n)},f\}_{n=1,2,\ldots}
$ be a sequence of measurable functions, $ \mu  $ be a measure on $\S$, and
%If %all the
%integrals $\int_{S}f(s) P^{(n)}(ds)$, $\int_{S}f^{(n)}(s) P^{(n)}(ds)$, $n=1,2,\ldots$, and $\int_{S}f(s) P(ds)$ are defined, and
%\begin{equation}\label{eq:cfl0}
%\slim\limits_{n\to\infty}\sup\limits_{S\in\Sigma}\left(\int_{S}f(s) P^{(n)}(ds)- \int_{S}f(s) P(ds)\right)\le 0,
%\end{equation}
 (\ref{eq:cfl2}) hold for each $\varepsilon>0$. Then there exists a subsequence $\{f^{(n_k)}\}_{k=1,2,\ldots}\subseteq\{f^{(n)}\}_{n=1,2,\ldots}$ %sequence
%$\{n_k\}_{k=1,2,\ldots}$
such that
(\ref{eq:cfl3}) holds.
\end{lemma}
\begin{proof}
Fix an arbitrary $\varepsilon>0$.
According to (\ref{eq:cfl2}), there exists a sequence
$\{n_k\}_{k=1,2,\ldots}$
such that\\ $ \mu(\S_{f-f^{(n_k)}\ge\varepsilon}) \le 2^{-k}$,
$k=1,2,\ldots\,.$ Thus,
\[
 \mu(\cup_{k=K}^{\infty}\S_{f-f^{(n_k)}\ge\varepsilon})\le
 \sum_{k=K}^{\infty} \mu(\S_{f-f^{(n_k)}\ge\varepsilon})\le \sum_{k=K}^{\infty} 2^{-k}\le 2^{-K+1},\]
$ K=1,2,\ldots\,.$ Therefore, $
\mu(\cap_{K=1}^{\infty}\cup_{k=K}^{\infty}\S_{f-f^{(n_k)}\ge\varepsilon})=0$,
that is, {for each $\varepsilon>0$}
\[
 \mu(\{s\in\S\,:\,\ilim\limits_{k\to\infty}f^{(n_k)}(s)\le f(s)-\varepsilon\})=0.
\]
 Thus, if (\ref{eq:cfl2}) holds for
each $\varepsilon>0$, then (\ref{eq:cfl3}) holds. \end{proof}

\begin{lemma}\label{lem:2}
Let the assumptions of Theorem~\ref{cfl} hold. Then inequality (\ref{eq:cfl1}) implies statement (i) from
Theorem~\ref{cfl}.
\end{lemma}
\begin{proof} On the contrary, if statement (i) from
Theorem~\ref{cfl} does not hold, then there exist a sequence
$\{n_k\to\infty\}_{k=1,2,\ldots}$
and positive constants $\varepsilon^*$ and $\delta^*$ such that
\begin{equation}\label{eq:ncfl1}
\mu(\S_{f-f^{(n_k)}\ge\varepsilon^*})\ge \delta ^*,\quad k=1,2,\ldots\ .
\end{equation}
Since the sequence of finite measures
$\{ \mu^{(n)}\}_{n=1,2,\ldots}$ converges in  total
variation to the finite measure $ \mu$, there exists $K_1= 1,2,\ldots,$ such
that
%\begin{equation}\label{eq:ncfl2}
%\[
%\sup\left\{|\int_\S f(s) P^{(n_k)}(ds)-\int_\S f(s) P(ds)|\ : \
%f:\S\to [-1,1]\mbox{ is  measurable} \right\} \le \frac{\delta^*}{4\varepsilon^*},
%\]
%%\end{equation}
%for each $k=K_1,K_1+1,\ldots\,.$ In particular,
\begin{equation}\label{eq:ncfl3}
\sup\limits_{S\in \Sigma}\left|\mu^{(n_k)}(S)- \mu(S) \right| \le
\frac{\delta^*}{4},\quad k=K_1,K_1+1,\ldots\,.\end{equation}
Therefore, inequalities (\ref{eq:ncfl1}) and (\ref{eq:ncfl3}) yield
that
\begin{equation}\label{eq:ncfl4}
\mu^{(n_k)}(\S_{f-f^{(n_k)}\ge\varepsilon^*})\ge \frac{3\delta ^*}{4},\quad k=K_1,K_1+1,\ldots\,.
\end{equation}
Let us set $C:=\int_\S|f(s)|\mu(ds)$.
Note that $C<\infty$, because $f\in L^1(\S;\mu)$. Chebyshev's
inequality yields
that $\mu(\S_{|f|\ge M})\le \frac{C}{M}$ for each $M>0$. %, in particular, for%
%$M=\frac{4C}{\delta^*}$.
 Thus, inequality (\ref{eq:ncfl3}) implies
\begin{equation}\label{eq:ncfl5}
 \mu^{(n_k)}(\S_{|f|\ge \frac{4C}{\delta^*}})\le \frac{\delta ^*}{2},\quad k=K_1,K_1+1,\ldots\,.
\end{equation}
 Moreover, inequalities
(\ref{eq:ncfl4}) and (\ref{eq:ncfl5}) yield
\begin{equation}\label{eq:ncfl6}
\mu^{(n_k)}(\S_{f-f^{(n_k)}\ge\varepsilon^*}\setminus \S_{|f|\ge
\frac{4C}{\delta^*}}) \ge \frac{\delta ^*}{4},\quad k=K_1,K_1+1,\ldots\,.
\end{equation}
 Indeed, for $k=K_1,K_1+1,\ldots,$
\[
\begin{aligned}
\frac{3\delta ^*}{4}&\le
\mu^{(n_k)}(\S_{f-f^{(n_k)}\ge\varepsilon^*})\le \mu^{(n_k)}(\S_{|f|\ge
\frac{4C}{\delta^*}})+
\mu^{(n_k)}(\S_{f-f^{(n_k)}\ge\varepsilon^*}\setminus \S_{|f|\ge
\frac{4C}{\delta^*}})\\ &\le \frac{\delta
^*}{2}+\mu^{(n_k)}(\S_{f-f^{(n_k)}\ge\varepsilon^*}\setminus
\S_{|f|\ge \frac{4C}{\delta^*}}), \end{aligned}\]
where the first inequality follows from
(\ref{eq:ncfl4}), the second inequality follows from subadditivity
of the finite measure $\mu^{(n_k)}$, and the third inequality
follows from (\ref{eq:ncfl5}). Inequality (\ref{eq:cfl1}) implies
the existence of $K_2= K_1,K_1+1,\ldots$ such that
\begin{equation}\label{eq:ncfl7}
\begin{aligned}\int_{\S_{f-f^{(n_k)}\ge \varepsilon^*}\setminus \S_{|f|\ge
\frac{4C}{\delta^*}}} & f^{(n_k)}(s) \mu^{(n_k)}(ds) \\ -&
\int_{\S_{f-f^{(n_k)}\ge\varepsilon^*}\setminus \S_{|f|\ge
\frac{4C}{\delta^*}}}f(s) \mu(ds)\ge -\frac{\varepsilon^*\delta
^*}{8} \end{aligned}
\end{equation}
for each $k=K_2,K_2+1,\ldots\,.$ The definition of
$\S_{f-f^{(n_k)}\ge\varepsilon^*}$ and inequalities (\ref{eq:ncfl6})
and (\ref{eq:ncfl7}) yield that for each $k=K_2,K_2+1,\ldots\,.$
\[
\begin{aligned}
-\frac{\varepsilon^*\delta ^*}{8}\le
&\int_{\S_{f-f^{(n_k)}\ge\varepsilon^*}\setminus \S_{|f|\ge
\frac{4C}{\delta^*}}}f(s) \mu^{(n_k)}(ds)\\ -&
\int_{\S_{f-f^{(n_k)}\ge\varepsilon^*}\setminus \S_{|f|\ge
\frac{4C}{\delta^*}}}f(s) \mu(ds) -
\varepsilon^*\mu^{(n_k)}(\S_{f-f^{(n_k)}\ge\varepsilon^*}\setminus \S_{|f|\ge \frac{4C}{\delta^*}})\\
\le & \int_{\S_{f-f^{(n_k)}\ge\varepsilon^*}\setminus \S_{|f|\ge
\frac{4C}{\delta^*}}}f(s) \mu^{(n_k)}(ds)-
\int_{\S_{f-f^{(n_k)}\ge\varepsilon^*}\setminus \S_{|f|\ge
\frac{4C}{\delta^*}}}f(s) \mu(ds) - \frac{\varepsilon^*\delta ^*}{4}.
\end{aligned}
\]
 Therefore, for each $k=K_2,K_2+1,\ldots,$
\begin{equation}\label{eq:ncfl8}
\begin{aligned}
\int_{\S}f(s)&{\bf I}\left\{s\in
\S_{f-f^{(n_k)}\ge\varepsilon^*}\setminus \S_{|f|\ge
\frac{4C}{\delta^*}}\right\} \mu^{(n_k)}(ds)\\&- \int_{\S}f(s) {\bf
I}\left\{s\in \S_{f-f^{(n_k)}\ge\varepsilon^*}\setminus \S_{|f|\ge
\frac{4C}{\delta^*}}\right\}\mu(ds)\ge \frac{\varepsilon^*\delta
^*}{8}.
\end{aligned}
\end{equation}
 Since each function $s\to f(s){\bf
I}\left\{s\in \S_{f-f^{(n_k)}\ge\varepsilon^*}\setminus \S_{|f|\ge
\frac{4C}{\delta^*}}\right\}$, $k=K_2,K_2+1,\ldots,$ is measurable
and absolutely bounded by the constant $\frac{4C }{\delta^*}$ and
the sequence of finite measures $\{ \mu^{(n)}\}_{n=1,2,\ldots}$ converges in total variation to $ \mu\in\M(\S)$,
\[
\begin{aligned}
\int_{\S}f(s)&{\bf I}\left\{s\in
\S_{f-f^{(n_k)}\ge\varepsilon^*}\setminus \S_{|f|\ge
\frac{4C}{\delta^*}}\right\} \mu^{(n_k)}(ds)\\&- \int_{\S}f(s) {\bf
I}\left\{s\in \S_{f-f^{(n_k)}\ge\varepsilon^*}\setminus \S_{|f|\ge
\frac{4C}{\delta^*}}\right\}\mu(ds)\to 0, \ k\to \infty.
\end{aligned}
\]
This contradics (\ref{eq:ncfl8}). Therefore,
inequality (\ref{eq:cfl1}) implies statement (i) of
Theorem~\ref{cfl}.
 %
% inequality (\ref{eq:ncfl8}) contradicts to (\ref{eq:ncfl2})
%
%
%Note that for all $n=1,2,\ldots$
%\begin{equation}\label{eq:cfl4}
%\begin{aligned}
%\varepsilon  P^{(n)}(S^{(n)}_{\varepsilon}) \le \int_{S^{(n)}_{\varepsilon}} f(s) P^{(n)}(ds) &-\int_{S^{(n)}_{\varepsilon}}
%f^{(n)}(s) P^{(n)}(ds) \\
%\le -\left(\, \int_{S^{(n)}_{\varepsilon}} f^{(n)}(s) P^{(n)}(ds)- \int_{S^{(n)}_{\varepsilon}} f(s) P(ds)\right)&  +\left(\,\int_{S^{(n)}_{\varepsilon}} f(s) P^{(n)}(ds) -\int_{S^{(n)}_{\varepsilon}}
%f(s) P(ds)\right).
%\end{aligned}
%\end{equation}
%Condition (\ref{eq:cfl0}) yields that
%\[
%\slim\limits_{n\to\infty}\left(\,\int_{S^{(n)}_{\varepsilon}} f(s) P^{(n)}(ds) -\int_{S^{(n)}_{\varepsilon}}
%f(s) P(ds)\right)\le 0.
%\]
%Therefore, inequalities (\ref{eq:cfl1}) and (\ref{eq:cfl4}) imply
%\begin{equation}\label{eq:cfl6as}
% P^{(n)}(S^{(n)}_{\varepsilon})\to 0\quad {\rm as} \quad n\to\infty.
%\end{equation}
%The convergence in total variation of $ P^{(n)}$ to $ P\in \M(\S)$  yields that
%\begin{equation}\label{eq:cfl5}
%| P^{(n)}(S^{(n)}_{\varepsilon})- P(S^{(n)}_{\varepsilon})|\to 0 {\rm\ as\ }
%n\to\infty.
%\end{equation}
%Thus, formulas (\ref{eq:cfl6as}) and (\ref{eq:cfl5}) imply
%\begin{equation}\label{eq:cfl6}
% P(S^{(n)}_{\varepsilon})\to 0\quad {\rm as} \quad n\to\infty,
%\end{equation}
%Therefore, (\ref{eq:cfl2}) holds for each $\varepsilon>0$.
\end{proof}

\begin{lemma}\label{lem:3}
Let $(\S,\Sigma)$ be a measurable space,  $\{
\mu^{(n)},\mu\}_{n=1,2,\ldots}\subset \M(\S)$, $f\in L^1(\S;\mu)$, and   $f^{(n)}
\in L^1(\S;\mu^{(n)})$, for each $n=1,2,\ldots\,.$
Then, inequality (\ref{eq:cfl1}) and statement (i) from
Theorem~\ref{cfl} imply statement (ii) from
Theorem~\ref{cfl}.
%Let the assumptions of Theorem~\ref{cfl} hold. Then inequality (\ref{eq:cfl1}) implies statement (ii) from
%Theorem~\ref{cfl}.
\end{lemma}
\begin{remark}\label{rem:lem:3}
{\rm According to Lemmas~\ref{lem:2} and \ref{lem:3}, if the assumptions of Theorem~\ref{cfl} hold, then inequality (\ref{eq:cfl1}) implies statements (i) and (ii) from
Theorem~\ref{cfl}.}
\end{remark}

\begin{proof}[Proof of Lemma~\ref{lem:3}.] For each $Q\in \M(\S)$ and $g\in L^1(\S;Q),$%\textbf{ [
%Chebyshev's inequality yields that
%\[%\begin{equation}\label{eq:ui0}
%Q(\{s\in\S\,:\, g(s)\le -K\}) \to 0\mbox{ as }K\to +\infty.
%\]%end{equation}
%Thus, ]}
\begin{equation}\label{eq:ui01}
\int_{\S_{g\le -K}}g(s) Q(ds)\to 0\mbox{ as }K\to +\infty.
\end{equation} Therefore, statement (ii) of Theorem~\ref{cfl} is
equivalent to the existence of a natural number $N$ such that for each $\varepsilon>0$
\begin{equation}\label{eq:ui000}
\ilim\limits_{K\to+\infty}\inf\limits_{n=N,N+1,\ldots
}\int_{\S_{f^{(n)}\le -K}}f^{(n)}(s) \mu^{(n)}(ds)\ge -\varepsilon.
\end{equation}

Let us fix an arbitrary $\varepsilon>0$ and verify (\ref{eq:ui000}).
{According to inequality (\ref{eq:cfl1}), there exists
$N_1=1,2,\ldots$ such that for  $n=N_1,N_1+1,\ldots$
\[
\inf\limits_{S\in\Sigma}\left(\int_{S}f^{(n)}(s)
\mu^{(n)}(ds)- \int_{S}f(s) \mu(ds)\right)\ge -\frac\varepsilon2.
\]
Then, for  $n=N_1,N_1+1,\ldots$ and $K>0,$}
\begin{equation}\label{eq:ui1}
%\begin{aligned}
\int_{\S_{f^{(n)}\le -K}} f^{(n)}(s) \mu^{(n)}(ds)\ge
\int_{\S_{f^{(n)}\le -K}}f(s) \mu(ds) -\frac\varepsilon2.
%\end{aligned}
\end{equation}
 Direct calculations imply that, for $ n=N_1,N_1+1,\ldots$ and for  $K>0,$
\begin{equation}\label{eq:ui2}
\begin{aligned}
&\int_{\S_{f^{(n)}\le -K}}f(s) \mu(ds) = \int_{ \S_{f-f^{(n)}<1}\cap
\S_{f^{(n)}\le -K}}f(s) \mu(ds)\\ & + \int_{ \S_{f-f^{(n)}\ge 1}\cap
\S_{f^{(n)}\le -K}}f(s)   \mu(ds)\ge -\int_{\S_{f\le 1-K}}|f(s)|
\mu(ds)\\ &-\int_{ \S_{f^{(n)}-f\le -1}}|f(s)| \mu(ds),
\end{aligned}
\end{equation}
{where the inequality holds because $\S_{f-f^{(n)}<1}\cap
\S_{f^{(n)}\le -K}\subseteq \S_{f\le 1-K}$ and $\S_{f-f^{(n)}\ge 1}\cap
\S_{f^{(n)}\le -K}\subseteq \S_{f^{(n)}-f\le -1}$.}
Due to (\ref{eq:ui01})
\begin{equation}\label{eq:ui3}
\int_{\S_{f\le -K+1}}|f(s)| \mu(ds) \to 0 \mbox{ as } K\to+\infty.
\end{equation}
%Since, {in view of Lemma~\ref{lem:2}, inequality (\ref{eq:cfl1}) yields statement (i) of
%Theorem~\ref{cfl}, then}
Statement (i) of
Theorem~\ref{cfl} yields that
$
\mu(\S_{ f^{(n)}- f\le -1})\to 0 \mbox{ as } n\to \infty.
$
Therefore, {since $f\in L^1(\S;\mu),$} there exists $N_2=N_1,N_1+1,\ldots$ such that
\begin{equation}\label{eq:ui4}
\int_{\S_{ f^{(n)}- f\le -1}}|f(s)|\mu(ds)\le \frac\varepsilon2, \quad n=N_2,N_2+1,\ldots\,.
\end{equation}
 Thus (\ref{eq:ui1}) -- (\ref{eq:ui4})
imply the existence of a natural number $N$ such that for each $\varepsilon>0$
(\ref{eq:ui000}) holds.
%
%
%
%for each $\varepsilon>0$
%\[
%\ilim\limits_{K\to+\infty}\inf\limits_{n=N_2,N_2+1,\ldots
%}\int_{\S_{f^{(n)}\le -K}}f^{(n)}(s) \mu^{(n)}(ds)\ge -2\varepsilon,
%\]
% that is  Statement (ii)' holds.
Therefore,
inequality (\ref{eq:cfl1}) and  statement (i) from
Theorem~\ref{cfl} imply statement (ii) {from}
Theorem~\ref{cfl}.\end{proof}

\begin{lemma}\label{lem:4}
Let the assumptions of Theorem~\ref{cfl} hold. Then statements (i) and (ii) {from} Theorem~\ref{cfl}
yield inequality (\ref{eq:cfl1}).
\end{lemma}
\begin{proof}{ Additivity of integrals and the property, that an infimum of a sum of two functions is greater than or equal to the sum of infimums, imply that,  for  $n=1,2,\ldots$ and $K>0,$}
\begin{equation}\label{eq:ui5}
\begin{aligned}
&\inf\limits_{S\in\Sigma}\left(\int_{S}f^{(n)}(s) \mu^{(n)}(ds)-
\int_{S}f(s) \mu(ds)\right)\\ &\ge
\inf\limits_{S\in\Sigma}\left(\int_{S_{f^{(n)}\le -K}}f^{(n)}(s)
\mu^{(n)}(ds)-\int_{S_{f^{(n)}\le -K}}f(s) \mu(ds)\right)\\
&+ \inf\limits_{S\in\Sigma}\left(\int_{S_{f^{(n)}> -K}}f^{(n)}(s)
\mu^{(n)}(ds)- \int_{S_{f^{(n)}> -K}}f(s)\mu(ds)\right),
\end{aligned}
\end{equation}
Note that,{ for  $n=1,2,\ldots$ and $K>0,$}
\begin{equation}\label{eq:ui6}
\begin{aligned}
&\inf\limits_{S\in\Sigma}\left(\int_{S_{f^{(n)}\le -K}}f^{(n)}(s)
\mu^{(n)}(ds)-\int_{S_{f^{(n)}\le -K}}f(s) \mu(ds)\right)\\
&\ge \inf\limits_{n=1,2,\ldots }\int_{\S_{f^{(n)}\le -K}}f^{(n)}(s)
\mu^{(n)}(ds)-\int_{\S_{f^{(n)}\le -K}}|f(s)| \mu(ds).
\end{aligned}
\end{equation}
 Moreover, {for  $n=1,2,\ldots$ and $K>0,$ }
\begin{equation}\label{eq:ui7}
\begin{aligned}
& \int_{\S_{f^{(n)}\le -K}}|f(s)| \mu(ds)\le \int_{\S_{f\le
-K+1}}|f(s)| \mu(ds)+ \int_{\S_{f-f^{(n)}\ge 1}}|f(s)| \mu(ds),
\end{aligned}
\end{equation}
because, if $f^{(n)}(s)\le -K$
and $f(s)> -K+1$, then $f^{(n)}(s)< f(s)-1$ and, thus,
$f^{(n)}(s)\le f(s)-1$.

Since $f\in L^1(\S;\mu)$, then $\mu(\S_{f\le -K+1})\to 0$ as
$K\to+\infty$. Therefore,
\begin{equation}\label{eq:ui8}
\int_{\S_{f\le -K+1}}|f(s)| \mu(ds)\to 0 \mbox{ as }K\to+\infty.
\end{equation}
Due to {(\ref{eq:cfl2}), $\mu(\S_{f-f^{(n)}\ge 1})\to 0$ as $n\to\infty.$   Similar to (\ref{eq:ui8}),}
\begin{equation}\label{eq:ui9}
\int_{\S_{f-f^{(n)}\ge 1}}|f(s)| \mu(ds)\to 0 \mbox{ as }n\to\infty.
\end{equation}
%because $f\in L^1(\S;\mu)$ and $\mu (\S_{f-f^{(n)}\ge 1})\to 0$ as $n\to\infty.$
According to (\ref{eq:ui5}) -- (\ref{eq:ui9}), inequality
(\ref{eq:cfl1}) follows from  statements (i) and (ii) of
Theorem~\ref{cfl}, if
\begin{equation}\label{eq:ui10}
\begin{aligned}
\ilim\limits_{K\to+\infty}\ilim\limits_{n\to\infty} I(n,K)\ge 0,
\end{aligned}\end{equation}
where, {for  $n=1,2,\ldots$  and $K>0,$} \[I(n,K):=\inf\limits_{S\in\Sigma}\left(\int_{S_{f^{(n)}>
-K}}f^{(n)}(s) \mu^{(n)}(ds)-\int_{S_{f^{(n)}> -K}}f(s) \mu(ds)\right).
\]
The rest of the proof establishes inequality (\ref{eq:ui10}).
We {observe} that for each $K>0$
\begin{equation}\label{eq:ui11}
\begin{aligned}
I(n,K)\ge I_1(n,K)+I_2(n,K)+I_3(n,K),\quad n=1,2,\ldots,
\end{aligned}\end{equation}
where
\[\begin{aligned}I_1(n,K)=&\inf\limits_{S\in\Sigma}\left(\int_{S_{|f^{(n)}|< K}}f^{(n)}(s)
\mu^{(n)}(ds)-\int_{S_{|f^{(n)}|< K}}f^{(n)}(s) \mu(ds)\right),\\
I_2(n,K)=&\inf\limits_{S\in\Sigma}\left(\int_{S_{|f^{(n)}|<
K}}f^{(n)}(s)
\mu(ds)-\int_{S_{|f^{(n)}|< K}}f(s) \mu(ds)\right),\\
I_3(n,K)=&\inf\limits_{S\in\Sigma}\left(\int_{S_{f^{(n)}\ge
K}}f^{(n)}(s) \mu^{(n)}(ds)-\int_{S_{f^{(n)}\ge K}}f(s) \mu(ds)\right).
%-\sup\limits_{S\in\Sigma}
%\int_{S}f(s){\bf I}\{s\in\S\,:\, f(s)-\varepsilon\ge f^{(n)}(s)> -K\} P(ds),\\
%I_3(n,K,\varepsilon)=& \inf\limits_{S\in\Sigma}\left(\int_{S}f(s) {\bf
%I}\{s\in\S\,:\, f^{(n)}(s)> \max\{-K,f(s)-\varepsilon\}\}P^{(n)}(ds)\right.\\ &\qquad-\left.
%\int_{S}f(s){\bf I}\{s\in\S\,:\, f^{(n)}(s)> \max\{-K,f(s)-\varepsilon\}\} P(ds)\right)-\varepsilon.
\end{aligned}\]

Since $\{\mu^{(n)}\}_{n=1,2,\ldots}$ converges in total variation
to $\mu$, then $I_1(n,K)\to 0$ as $n\to \infty$ for each $K>0$.
Therefore,
\begin{equation}\label{eq:ui11a}
\begin{aligned}
\ilim\limits_{K\to+\infty}\ilim\limits_{n\to\infty} I_1(n,K)= 0.
\end{aligned}\end{equation}

{For}   $n=1,2,\ldots,$ $K>0$, and $\varepsilon>0$, the following inequalities hold:
\[
\begin{aligned}
I_2(n,K)&\ge
\inf\limits_{S\in\Sigma}\int_{S_{|f^{(n)}|< K}\cap
S_{f^{(n)}-f>-\varepsilon}}\left(f^{(n)}(s)-f(s)\right) \mu(ds)\\
&+\inf\limits_{S\in\Sigma}\int_{S_{|f^{(n)}|< K}\cap
S_{f-f^{(n)}>\varepsilon}}\left(f^{(n)}(s)-f(s)\right) \mu(ds)\\
&\ge -\varepsilon\mu(\S)-\int_{\S_{|f^{(n)}|< K}\cap
\S_{f-f^{(n)}>\varepsilon}} |f^{(n)}(s)| \mu(ds) -\int_{
\S_{f-f^{(n)}>\varepsilon}} |f(s)| \mu(ds).
\end{aligned}
\]
and, therefore,
\[
I_2(n,K) \ge  -\varepsilon\mu(\S) -K\mu(S_{f-f^{(n)}>\varepsilon})-\int_{
\S_{f-f^{(n)}>\varepsilon}}|f(s)| \mu(ds).
\]
Thus, due to
(\ref{eq:cfl2}) and $f\in L^1(\S;\mu)$,
\begin{equation}\label{eq:ui11b}
\begin{aligned}
\ilim\limits_{K\to+\infty}\ilim\limits_{n\to\infty} I_2(n,K)\ge 0.
\end{aligned}\end{equation}

{For  $n=1,2,\ldots$ and $K>0,$} the following {inequalities hold:}
\[\begin{aligned}
&I_3(n,K) \ge K\left(\mu^{(n)}(S_{f^{(n)}\ge K}\cap S_{f(s)\le
K})-\mu(S_{f^{(n)}\ge K}\cap S_{f(s)\le K})\right)\\
&-\sup\limits_{S\in\Sigma}\int_{S_{f^{(n)}\ge K}\cap S_{f(s)>K}}f(s)
\mu(ds)\ge -K\sup_{S\in
\Sigma}|\mu^{(n)}(S)-\mu(S)|-\int_{\S_{f(s)>K}}f(s) \mu(ds).
\end{aligned}
\]
Therefore, since
the sequence $\{\mu^{(n)}\}_{n=1,2,\ldots}$
converges in total variation to $\mu$ and $f\in L^1(\S)$, then
\begin{equation}\label{eq:ui11c}
\begin{aligned}
\ilim\limits_{K\to+\infty}\ilim\limits_{n\to\infty} I_3(n,K)\ge 0.
\end{aligned}\end{equation}
%
%\begin{equation}\label{eq:ui12}
%\begin{aligned}
%I_1&(n,K,\varepsilon)\ge  -K\left(P^{(n)}(\{s\in\S\,:\, f(s)-\varepsilon\ge f^{(n)}(s)\})\right.\\ &\left. -P(\{s\in\S\,:\, f(s)-\varepsilon\ge f^{(n)}(s)\})\right)
%-KP(\{s\in\S\,:\, f(s)-\varepsilon\ge f^{(n)}(s)\}),\\
%I_2&(n,K,\varepsilon)\ge-
%\int_{S}|f(s)|{\bf I}\{s\in\S\,:\, f(s)-\varepsilon\ge f^{(n)}(s)\} P(ds),\\
%I_3&(n,K,\varepsilon)\ge \inf\limits_{S\in\Sigma}\left(\int_{S}f(s) {\bf
%I}\{s\in\S: f^{(n)}(s)> \max\{-K,f(s)-\varepsilon\}\}P^{(n)}(ds)\right.\\ &-\left.
%\int_{S}f(s){\bf I}\{s\in\S\,:\, f^{(n)}(s)> \max\{-K,f(s)-\varepsilon\}\} P(ds)\right)-\varepsilon.
%\end{aligned}\end{equation}

Inequalities (\ref{eq:ui11})--(\ref{eq:ui11c}) yield
(\ref{eq:ui10}). Therefore, statements (i) and (ii) of
Theorem~\ref{cfl} {imply} inequality (\ref{eq:cfl1}).\end{proof}

\begin{proof}[Proof of Theorem~\ref{cfl}] Theorem~\ref{cfl} {follows directly} from Lemmas~\ref{lem:1}--\ref{lem:4}; see also Remark~\ref{rem:lem:3}.
\end{proof}

\begin{proof}[Proof of Corollary~\ref{cor:cflnongen}]
Corollary~\ref{cor:cflnongen} {follows directly} from Theorem~\ref{cfl} because inequality (\ref{eq:ui}) holds for the sequence of nonnegative functions $\{f^{(n)}\}_{n=1,2,\ldots}$.
\end{proof}

\begin{proof}[Proof of Corollary~\ref{cor:cfl}]
Corollary~\ref{cor:cfl} {follows directly} from Theorem~\ref{cfl}.
\end{proof}

\begin{proof}[Proof of Corollary~\ref{cor:cflnon}]
Corollary~\ref{cor:cflnon} {follows directly} from Corollary~\ref{cor:cflnongen}.
\end{proof}

\begin{proof}[Proof of Corollary~\ref{cdct}] Theorem~\ref{cfl}, being applied to the functions $\{f,f^{(n)}\}_{n=1,2,\ldots}$ and $\{-f,-f^{(n)}\}_{n=1,2,\ldots}$, yields %all the statements of
{Corollary~\ref{cdct}}.
\end{proof}

\begin{proof}[Proof of Corollary~\ref{cor:cdct}] Corollary~\ref{cor:cfl}, being applied to the functions $\{f,f^{(n)}\}_{n=1,2,\ldots}$ and $\{-f,-f^{(n)}\}_{n=1,2,\ldots}$, %yields all the statements of {Corollary~\ref{cor:cdct}}.
\end{proof}

\begin{proof}[Proof of Corollary~\ref{cor:RNin}] {If $\nu\in \M(\S)$, $\tilde{\nu}$ be a finite signed measure on $\S$, and ${ \Tilde \nu}\ll\nu$, then the Radon-Nikodym derivative $\frac{d{ \Tilde \nu}}{d\nu}$ is $\mu$-integrable, that is, $\frac{d{ \Tilde \nu}}{d\nu}\in L^1(\S;\nu)$. This is true because $\int_\S|\frac{d{ \Tilde \nu}}{d\nu}|d\nu=\|\nu\|<\infty.$} Set $f:=\frac{d{ \Tilde \mu}}{d\mu}$, $f^{(n)}:=\frac{d{ \Tilde \mu}^{(n)}}{d\mu^{(n)}},$ $n=1,2,\ldots\,$. Then Theorem~\ref{cfl} yields %all  the  statements of
{Corollary~\ref{cor:RNin}}.
\end{proof}

\begin{proof}[Proof of Corollary~\ref{cor:RNeq}] Corollary~\ref{cor:RNin}, being applied to  $\{{ \Tilde \mu},{ \Tilde \mu}^{(n)},\mu,\mu^{(n)}\}_{n=1,2,\ldots}$ and\\ $\{-{ \Tilde \mu},-{ \Tilde \mu}^{(n)},\mu,\mu^{(n)}\}_{n=1,2,\ldots}$, yields % all the statements of
{Corollary~\ref{cor:RNeq}}.
\end{proof}

\section{Counterexamples}

Example~\ref{exa0} describes a probability space $(\S,\Sigma,\mu)$ and a sequence $\{f,f^{(n)}\}_{n=1,2,\ldots}$ of uniformly bounded nonnegative measurable functions on it such that: (a) $\{f,f^{(n)}\}_{n=1,2,\ldots}$ satisfy inequality
(\ref{eq:cor:cfl1}); (b) inequality (\ref{eq:cfl2}) takes place for each $\varepsilon>0$; (c)
inequality (\ref{eq:cfl3}) does not hold for the function $f$ and the entire sequence
$\{f^{(n)}\}_{n=1,2,\ldots}$.  This example also demonstrates that
Corollary~\ref{cor:cflnon} is essentially a more exact statement {than} the classic  Fatou's lemma.

\begin{example}\label{exa0}
{\rm Let $\S=[0,1]$, $\Sigma$
be the Borel $\sigma$-field on $\S$,  $\mu^{(n)}=\mu$ {be} the Lebesgue
measure on $\S$,  $f\equiv1$, and  $f^{(n)}(s)=1-{\bf I}\{s\in
[\frac{j}{2^k}, \frac{j+1}{2^k}]\}$, where $k=[\log_2n]$, $j=n-2^k$,
$s\in\S$, and $n=1,2,\ldots\,.$ Then
\[
\lim\limits_{n\to\infty}\int_{\S}(f^{(n)}(s)- f(s) )^-\mu(ds)=\lim\limits_{n\to\infty}\frac{1}{2^{[\log_2n]}}=0,
\]
and, according to Remark~\ref{rem -}, inequality (\ref{eq:cor:cfl1}) holds. Moreover,
for each $\varepsilon>0$
\[
 \mu(\{s\in\S\,:\, f^{(n)}(s)\le f(s)-\varepsilon\})=\frac{1}{2^{[\log_2n]}} \to 0\mbox{ as }n\to\infty,
\]
that is, convergence in (\ref{eq:cfl2}) takes place for each $\varepsilon>0$.
Moreover,
\[\ilim\limits_{n\to\infty}f^{(n)}(s)=0< 1=f(s)\mbox{ for }\mu\mbox{-a.e. }s\in\S,
\]
that is, inequality (\ref{eq:cfl3}) does not hold for the function $f$ and for the entire sequence
$\{f^{(n)}\}_{n=1,2,\ldots}$.

Corollary~\ref{cor:cflnon} {yields} %that
 \[
1= \ilim\limits_{n\to\infty}\int_{\S}f^{(n)}(s)
\mu(ds)\ge \int_{\S}f(s) \mu(ds)=1;
 \]
see equality (\ref{eq:cor:cfl1}) and Remark~\ref{rem -}. But the classic  Fatou's lemma implies
\[
1= \ilim\limits_{n\to\infty}\int_{\S}f^{(n)}(s)
\mu(ds)\ge \int_{\S}\ilim\limits_{n\to\infty}f^{(n)}(s) \mu(ds)=0.
\]
Therefore, Corollary~\ref{cor:cflnon} is a more exact statement {than} the classic  Fatou's lemma. \hfill$\Box$
 }
\end{example}
%If the absolute value $|f|$ of a function $f$ from
%Theorem~\ref{cfl} is bounded above by a constant $C>0$  on $\S$,
%then assumption (\ref{eq:cfl0}) holds, because the sequence $\{
%P^{(n)}\}_{n=1,2,\ldots}\subset \M(\S)$ converges in  total
%variation to
%$ P\in\M(\S)$. Moreover, $f\in L^1(\S;P)\cap L^1(\S;P^{(n)})$ for each $n=1,2,\ldots$.}

{The following {three} examples demonstrate that the uniform Fatou's lemma does not hold, if convergence of measures in total variation is relaxed to
setwise convergence.  In particular,  the necessary condition fails in {Examples~\ref{exa2} and \ref{exa4}}, and   the sufficient condition fails in Example~\ref{exa3}.  As mention above, Fatou's lemma, which is a sufficient condition for inequality \eqref{eq:FLcl1}, which is weaker that inequality \eqref{eq:cfl1} in the uniform Fatou's lemma, holds for setwise converging measures and, if the notion of a limit of a function is appropriately modified, it also holds for weakly converging measures; see Royden~\cite[p. 231]{Ro}, Serfozo~\cite{Se},  Feinberg et al.~\cite{TVP}, and references therein.}

{Example~\ref{exa2} demonstrates that,  if  convergence in total
variation of finite measures $\{ \mu^{(n)}\}_{n=1,2,\ldots}$ to $
\mu$ in Corollary~\ref{cdct} is relaxed to  setwise convergence, equality (\ref{eq:cdct1}) { implies neither statement (i) nor statement (ii) from
Theorem~\ref{cfl}, and therefore neither statement (i) nor statement (ii) from  Corollary~\ref{cdct} holds}.
Thus, inequality (\ref{eq:cfl1}) does not yield {either
 statement (i) or statement (ii)} from Theorem~\ref{cfl}, if the convergence in total
variation of finite measures $\{ \mu^{(n)}\}_{n=1,2,\ldots}$ to $
\mu$ in
%either
Theorem~\ref{cfl} %or Corollary~\ref{cor:cflnongen}
is relaxed to   setwise convergence.}

\begin{example}\label{exa2}
{\rm
Let $\S=[0,1]$,  $\Sigma=\B(\S)$ be a Borel $\sigma$-algebra
on $\S$,
\begin{align*}
g^{(n)}(s):=\begin{cases} \frac1n, &\textrm{if $2k/2^n<s<(2k+1)/2^n$ for
$k=0,1,\ldots,2^{n-1}-1$;}\\
2-\frac1n, &\textrm{otherwise,}
\end{cases}
\end{align*}
$f^{(n)}(s):=-1/g^{(n)}(s)$, $s\in[0,1]$, $n=1,2,\ldots,$ be the sequence of
measurable functions,  $ \mu$ be the Lebesgue measure on $[0,1]$,
and  $f\equiv -1$. Consider the sequence of probability measures $
\mu^{(n)}$  on $[0,1]$, $n=1,2,\ldots,$ defined as
\begin{equation}\label{eq:exa22}
 \mu^{(n)}(S):=\int_{S}g^{(n)}(s) \mu(ds),\quad S\in \Sigma.
\end{equation}

The sequence $\{ \mu^{(n)}\}_{n=1,2,\ldots}$ converges setwise to $ \mu$
as $n\to\infty$.
Indeed, {according to Feinberg et al.~\cite[Theorem 2.3]{POMDP2},}
%Bogachev \cite[Theorem 8.10.56]{bogachev}, which is Pflanzagl's
%generalization of the Fichtengolz-Dieudonn\'e-Grothendiek theorem,
measures $\mu^{(n)}$ converge setwise to the measure $\mu$, if
$\mu^{(n)}(C)\to \mu(C)$ for each open set $C$ in $[0,1]$.  Since
$\mu^{(n)}(\{0\})=\mu(\{0\})=\mu^{(n)}(\{1\})=\mu(\{1\}),$ $n=1,2,\ldots,$ then
$\mu^{(n)}(C)\to \mu(C)$ for each open set $C$ in $[0,1]$ if and only
if $\mu^{(n)}(C)\to \mu(C)$ for each open set $C$ in $(0,1).$
 Choose an arbitrary open set $C$ in $(0,1).$ Then $C$ is a union of a
countable  set of open disjoint intervals $(a_i,b_i)$. Therefore, for each $\varepsilon>0$ there is a finite number $n_\varepsilon$ of open intervals
$\{(a_i,b_i):i=1,\ldots,n_\varepsilon\}$ such that $\mu(C\setminus C_\varepsilon)\le \varepsilon$, where $C_\varepsilon=\cup_{i=1}^{n_\varepsilon}
(a_i,b_i).$ Due to $|g^{(n)}|\le 2$, we obtain that $\mu^{(n)}(C\setminus C_\varepsilon)\le 2\varepsilon$ for each $n=1,2,\ldots\ .$ Since
$|\mu^{(n)}((a,b))-\mu((a,b))|< 1/2^{n-1}, $ $n=1,2,\ldots,$ for each interval $(a,b)\subseteq (0,1),$  this implies that
$|\mu(C_\varepsilon)-\mu^{(n)}(C_\varepsilon)|< \varepsilon$ if $n\ge N_\varepsilon$, where $N_\varepsilon$ is each natural number
satisfying $1/2^{N_\varepsilon-1}\le \varepsilon.$ %EF10282013
Therefore, if $n\ge N_\varepsilon$ then $|\mu^{(n)}(C)-\mu(C)|\le |\mu^{(n)}(C_\varepsilon)-\mu(C_\varepsilon)|+\mu(C\setminus C_\varepsilon)+\mu^{(n)}(C\setminus
C_\varepsilon)< 4\varepsilon.$  This implies that $\mu^{(n)}(C)\to \mu(C)$ as $n\to\infty$.  Thus $\mu^{(n)}$ converge setwise to $\mu$ as $n\to\infty$.

{Observe that for $S_n=\cup_{k=0}^{2^{n-1}-1}\,[2k/2^n,(2k+1)/2^n],$ $n=1,2,\ldots,$
\begin{equation}\label{eq:negg}
\mu^{(n)}(S_n)-\mu(S_n)=-(\frac{1}{2}-\frac{1}{2n}).
\end{equation}
So, the sequence $\{ \mu^{(n)}\}_{n=1,2,\ldots}$ does not converge in total variation to $ \mu$
%as $n\to\infty$
because
\[
{\rm dist}( \mu^{(n)},\mu)\ge \frac{1}{2}-\frac{1}{2n},\quad n=1,2,\ldots\ .
\]}

{Equality (\ref{eq:cdct1}) holds since
\begin{equation}\label{eq3.3accur}
\int_{S}f^{(n)}(s) \mu^{(n)}(ds)=
\int_{S}f(s) \mu(ds)\quad\ \textrm{for all}  \ S\in\Sigma, \ n=1,2,\ldots,
\end{equation}
which is stronger than {(\ref{eq:cdct1})}.
 Thus, inequality
(\ref{eq:cfl1})  also holds.}

Statement (i) from
Theorem~\ref{cfl} does not hold since
 \[
 \mu(\{s\in\S\,:\, f^{(n)}(s)\le f(s)-1\})=\frac12, \quad
n=2,3,\ldots\ . \]
Thus statement (i) from  Corollary~\ref{cdct} does not hold either.

Statement (ii) from
Theorem~\ref{cfl} does not hold since
 \[
 \inf\limits_{n=1,2,\ldots
}\int_{\S}f^{(n)}(s){\bf I}\{s\in\S\,:\, f^{(n)}(s)\le -K\}
\mu^{(n)}(ds)=\frac12,\quad K>1. \]
Thus statement (ii) from  Corollary~\ref{cdct}  does not hold either. \hfill$\Box$

%\textbf{[ Pasha, I think that this paragraph should be taken out.  This is the same text as the paragraph preceding the example.  I have edited it there.  While I was editing, I deleted mentioning Corollary 1.3.  If you think that it should be mentioned, please do either in that paragraph or write a sentence here.*********** Therefore, equality (\ref{eq:cdct1}) does not imply statements (i) and (ii) of
%Theorem~\ref{cfl} (thus, statements (i) and (ii) of  Corollary~\ref{cdct} are also do not hold), if the convergence in total
%variation of finite measures $\{ \mu^{(n)}\}_{n=1,2,\ldots}$ to $
%\mu$ in Corollary~\ref{cdct} is relaxed to  the setwise convergence.
%Therefore, inequality (\ref{eq:cfl1}) does not yield
% statements (i) and (ii) from Theorem~\ref{cfl}, if the convergence in total
%variation of finite measures $\{ \mu^{(n)}\}_{n=1,2,\ldots}$ to $
%\mu$ in either Theorem~\ref{cfl} or Corollary~\ref{cor:cflnongen} is relaxed to  the setwise convergence.]}
}
\end{example}

Example~\ref{exa4} demonstrates that,  if  convergence in total
variation of finite measures $\{ \mu^{(n)}\}_{n=1,2,\ldots}$ to $
\mu$ in Corollary~\ref{cor:cflnongen}, {in which the functions $f^{(n)}$ are assumed to be nonnegative,} is relaxed to  setwise convergence, inequality (\ref{eq:cfl1}) does not imply statement (i) from
Theorem~\ref{cfl}.

\begin{example}\label{exa4}
{\rm Let $\S=[0,1]$,  $\Sigma=\B(\S)$ be a Borel $\sigma$-algebra
on $\S$,
\begin{align}\label{eq3.3forg}
g^{(n)}(s):=\begin{cases} \frac12, &\textrm{if $2k/2^n<s<(2k+1)/2^n$ for
$k=0,1,\ldots,2^{n-1}-1$;}\\
\frac32, &\textrm{otherwise,}
\end{cases}
\end{align}
$f^{(n)}(s):=1/g^{(n)}(s)$, $s\in[0,1]$, $n=1,2,\ldots,$ be the sequence of
measurable functions,  $ \mu$ be the Lebesgue measure on $[0,1]$,
and  $f\equiv 1$. {For the functions $g^{(n)}$ from \eqref{eq3.3forg}, consider the sequence of probability measures $
\mu^{(n)}$  on $[0,1]$,  $n=1,2,\ldots,$ defined in (\ref{eq:exa22}).}

{The sequence $\{ \mu^{(n)}\}_{n=1,2,\ldots}$ converges setwise to $ \mu$
as $n\to\infty,$  and \eqref{eq3.3accur} holds. These facts follows from the same arguments as in Example~\ref{exa2}. In view of
\eqref{eq3.3accur}, inequality (\ref{eq:cfl1}) holds.}
 %the following equalities hold:
%\[
%\sup\limits_{S\in\Sigma}\left|\int_{S}f^{(n)}(s)
%\mu^{(n)}(ds)- \int_{S}f(s) \mu(ds)\right|=\int_{\S}\left|f^{(n)}(s)g^{(n)}(s)-f(s)\right|
%\mu(ds)=0,
%\]
%and this equality implies inequality (\ref{eq:cfl1}).
%
Statement (i) from
Theorem~\ref{cfl} does not hold since
 $\mu(\{s\in\S\,:\, f^{(n)}(s)\le f(s)-\frac13\})=\frac12$,
$n=1,2,\ldots\ .$
\hfill$\Box$}\end{example}

Example~\ref{exa3} demonstrates that
statements (i) and (ii) from Corollary~\ref{cdct} do not imply inequality (\ref{eq:cfl1}) {and therefore they do not imply
equality (\ref{eq:cdct1}), if  convergence} in total
variation of finite measures $\{ \mu^{(n)}\}_{n=1,2,\ldots}$ to $
\mu$ in Corollary~\ref{cdct} is relaxed {to setwise }convergence. Therefore, statements (i) and (ii) from Theorem~\ref{cfl}
do not yield inequality (\ref{eq:cfl1}), if the convergence in total
variation of finite measures $\{ \mu^{(n)}\}_{n=1,2,\ldots}$ to $
\mu$ in  Theorem~\ref{cfl}  is relaxed {to setwise} convergence.

\begin{example}\label{exa3}
{\rm
Consider a measurable space $(\S,\Sigma)$ and a sequence $\{\mu^{(n)}\}_{n=1,2,\ldots}\subset \M(\S)$ that converges setwise to a measure $\mu$ on $\S$ such that \[
\ilim_{n\to\infty}\inf_{S\in \Sigma} \left(\mu^{(n)}(S)-\mu(S) \right)<0.
\]
For example, in view of \eqref{eq:negg}, the measurable spaces and measures defined in Example~\ref{exa2} can be considered for this example. Let $f=f^{(n)}\equiv 1$, $n=1,2,\ldots.$

Note that, statements (i) and (ii) from  Corollary~\ref{cdct} hold. Thus statements (i) and (ii) from  Theorem~\ref{cfl} hold.
Moreover, since
\[
\ilim\limits_{n\to\infty}\inf\limits_{S\in\Sigma}\left(\int_{S}f^{(n)}(s)
\mu^{(n)}(ds)- \int_{S}f(s) \mu(ds)\right)=\ilim_{n\to\infty}\inf_{S\in \Sigma} \left(\mu^{(n)}(S)-\mu(S) \right)<0,
\]
then neither inequality (\ref{eq:cfl1}) nor equality (\ref{eq:cdct1}) holds. \hfill$\Box$
%Therefore, the statements of Theorem~\ref{cfl} and Corollaries~\ref{cor:cflnongen} and \ref{cdct} do not hold.
%\textbf{[Pasha, this can be deleted ********************************Therefore,
%statements (i) and (ii) of  Corollary~\ref{cdct} do not imply inequality (\ref{eq:cfl1}) (thus,
%equality (\ref{eq:cdct1}) is also does not hold),  if the convergence in total
%variation of finite measures $\{ \mu^{(n)}\}_{n=1,2,\ldots}$ to $
%\mu$ in Corollary~\ref{cdct} is relaxed to  the setwise convergence. Furthermore, statements (i) and (ii) from Theorem~\ref{cfl}
%do not yield inequality (\ref{eq:cfl1}), if the convergence in total
%variation of finite measures $\{ \mu^{(n)}\}_{n=1,2,\ldots}$ to $
%\mu$ in either Theorem~\ref{cfl} or Corollary~\ref{cor:cflnongen} is relaxed to  the setwise convergence.]}
 }
\end{example}

{ We remark
that  the functions $f$ and $f^{(n)},$ $n=1,2,\ldots,$ are nonnegative in  Example~\ref{exa3}.  Therefore, unlike the case of measures
converging in  total variation described in Corollary~\ref{cdct}, even for nonnegative functions $f$ and $f^{(n)},$ the validity of statement (i) from Theorem~\ref{cfl} is not necessary for the validity of inequality (\ref{eq:cfl1}) in the case of setwise converging measures.}

%The following example provides that assumption (\ref{eq:cfl0}) is essential in Theorem~\ref{cfl}.
%
%\begin{example}\label{exa3}
%{\rm
%Let $\S=\{1,2,\ldots,n,n+1,\ldots \},$ $\Sigma=2^\S$,
%\[
%P(\{s\})=\left\{
%\begin{array}{ll}
%\frac23,& s=1;\\
%\frac1{4^s},&\mbox{otherwise};
%\end{array}
%\right.\quad P^{(n)}(\{s\})=\left\{
%\begin{array}{ll}
%P(\{s\})+\frac17 , & s=n;\\
%P(\{s\})-\frac{1}{7(n-1)} {\bf I} \{s\le n-1\} , & \mbox{otherwise},
%\end{array}
%\right.
%\]
%$s\in \S$, $n=1,2,\ldots,$
%}

\vspace{.3cm}
 {\bf Acknowledgement.}  Research of the first author was
partially supported by NSF grant CMMI-1335296.


\begin{thebibliography}{00}


%\bibitem{bogachev} \textbf{Bogachev V.I. (2007) \textit{Measure Theory, Volume II}
%(Springer-Verlag, Berlin).}

%\bibitem{GKK} Gorban N.V., Kapustyan O.V., Kasyanov P.O. (2014) Uniform trajectory attractor for non-autonomous reaction–diffusion equations with Carath\'eaodory's nonlinearity.
%\emph{Nonlinear Analysis: Theory, Methods $\&$ Applications} 98: 13--26.

%\bibitem{FKZ}  Feinberg EA,  Kasyanov PO, Zadoianchuk NV (2012)   Average-cost Markov decision
%processes with weakly continuous transition probabilities.
%\emph{Math. Oper. Res.} 37(4): 591--607.

\bibitem{TVP} E.A.~Feinberg,  P.O.~Kasyanov, N.V.~Zadoianchuk,  Fatou's lemma for weakly converging probabilities, Theory Probab. Appl. 58 (2014)  683--689. % doi: 10.4213/tvp4544

%\bibitem{IEEE}  Feinberg E.A., Kasyanov P.O., Zgurovsky M.Z. (2013) Optimality Conditions for Total-Cost Partially Observable Markov Decision Processes, 52nd IEEE Conference on Decision and Control, December 10-13, 2013. Florence, Italy, pp. 5716--5721

%\bibitem{POMDP1} Feinberg E.A., Kasyanov P.O., Zgurovsky M.Z. (2014) Optimality conditions for partially observable Markov decision processes, in \emph{Continuous and Distributed Systems.  Theory and Applications} (eds. M.Z. Zgurovsky, V.A. Sadovnichy), Springer, New York, 251--264.


\bibitem{POMDP} E.A.~Feinberg, P.O.~Kasyanov, M.Z.~Zgurovsky, Partially observable total-cost Markov decision processes with weakly continuous transition
probabilities, Mathematics of Operations Research (to appear), arXiv:1401.2168.


\bibitem{POMDP2} E.A.~Feinberg, P.O.~Kasyanov, M.Z.~Zgurovsky,  Convergence of probability measures and Markov decision models with incomplete information, Proceedings of the Steklov Institute of Mathematics 287 (2014) 96--117.

%\bibitem{in}  Jaskiewicz A,  Nowak AS (2006) Zero-sum ergodic stochastic games with Feller
%transition probabilities. \textit{SIAM J. Control Optim.} 45(3):
%773--789.

%\bibitem{Part}  Parthasarathy KR (1967) \textit{Probability Measures on Metric Spaces}. (Academic Press, New York).
%
\bibitem{Ro} H.L. Royden,
Real Analysis, Second edition, Macmillan, New York, 1968.

\bibitem{Se} R. Serfozo,   Convergence of Lebesgue integrals with
varying measures, Sankhya - The Indian Journal of
Statistics, Series A  44 (1982) 380--402.

%\bibitem{Sh}
%A.N. Shiryaev,  second edition,
%Springer-Verlag, New York, 1996.

%\bibitem{Schal} Sch\"{a}l M (1993) Average
%optimality in dynamic programming with general state space.
%\textit{Math. Oper. Res}. 18(1): 163--172.

%\bibitem{YL}  Y\"uksel S,  Linder T (2012) Optimization and convergence of observation channels in stochastic control.
%\textit{SIAM J. Control Optim.} 50(2): 864--887.
\end{thebibliography}
\end{document}